\pdfoutput=1
\RequirePackage[l2tabu, orthodox]{nag}

\documentclass[reqno,12pt,a4paper]{amsart}

\usepackage{lmodern}
\usepackage[T1]{fontenc}
\usepackage[utf8]{inputenc}
\usepackage[english]{babel}
\usepackage{microtype} %Better font spacing

\usepackage{amsmath,amssymb,amsthm,mathrsfs,latexsym,mathtools,mathdots,booktabs,enumerate,tikz,bm,url}
\usepackage[enableskew,vcentermath]{youngtab}
\usepackage[centertableaux]{ytableau}
\usepackage{enumitem} 

\usetikzlibrary{arrows,positioning}
\usepackage[capitalize,noabbrev]{cleveref}

\usepackage{todonotes}
\presetkeys%
    {todonotes}%
    {inline,backgroundcolor=yellow}{}

\usepackage{graphicx}

\usepackage{ifpdf}
\graphicspath{{./figures/}{./}}
\usepackage{epstopdf}
\DeclareGraphicsExtensions{.png,.jpg,.eps,.epsf,.pdf}

\newtheorem{theorem}{Theorem}
\newtheorem{proposition}[theorem]{Proposition}
\newtheorem{lemma}[theorem]{Lemma}
\newtheorem{corollary}[theorem]{Corollary}
\newtheorem{conjecture}[theorem]{Conjecture}

\theoremstyle{definition}
\newtheorem{definition}[theorem]{Definition}
\newtheorem{example}[theorem]{Example}
\newtheorem{remark}[theorem]{Remark}

\newcommand{\head}[1]{{\mathcal{H}(#1)}}
\newcommand{\tail}[1]{{\mathcal{T}(#1)}}

\newcommand{\setC}{\mathbb{C}}

\newcommand{\wvec}{\mathbf{w}}

\newcommand{\vvec}{\mathbf{v}}

\newcommand{\Vol}{\mathrm{Vol}}

\newcommand{\symS}{S}

\newcommand{\BST}{\mathrm{BST}}
\newcommand{\BSD}{\mathrm{BSD}}
\newcommand{\BSP}{\mathrm{BSP}}

\DeclareMathOperator{\length}{\ell}

\DeclareMathOperator{\inv}{inv}
\DeclareMathOperator{\hor}{hor}

\DeclareMathOperator{\DES}{DES}
\DeclareMathOperator{\des}{des}

\title[Enumeration of border-strip decompositions]{Enumeration of border-strip decompositions \& Weil--Petersson volumes}

\author{Per Alexandersson}
\address{Dept. of Mathematics, Royal Institute of Technology, SE-100 44 Stockholm, Sweden}
\email{per.w.alexandersson@gmail.com}

\author{Linus Jordan}
\address{Dept. of Mathematics, Royal Institute of Technology, SE-100 44 Stockholm, Sweden}
\email{linus.jordan@bluewin.ch}

\keywords{Border-strip tableaux, border-strip decompositions, permutations, $q$-analogue, Weil--Petersson volume}

\begin{document}
\begin{abstract}
We describe an injection from border-strip decompositions of certain shapes to permutations.
This allows us to provide enumeration results, as well as $q$-analogues of enumeration formulas.

Finally, we use this injection to prove a connection between the number of 
border-strip decompositions of the $n\times 2n$ rectangle and the Weil--Petersson volume
of the moduli space of an $n$-punctured Riemann sphere.
\end{abstract}

\maketitle

% \todo{
% On the number of rim-hook tableaux, 
% %\url{http://citeseerx.ist.psu.edu/viewdoc/download?doi=10.1.1.12.465&rep=rep1&type=pdf}
% NOTE: THESE count tableaux, NOT decompositions, but it gives an upper bound!}

\tableofcontents

\section{Introduction}

Border-strip \emph{tableaux} have a rich history, originating with the 
celebrated Murnaghan--Nakayama rule, \cite{Murnaghan1937,Nakayama1940}, 
which provides a combinatorial formula 
for computing character values of $\symS_n$.
It is a signed sum over border-strip tableaux, but the sign only depends on
the border-strip \emph{decomposition}, \emph{i.e.,} the ``unlabeled version''
of the tableaux. This gives a motivation to enumerate border-strip decompositions.

We note that there is a hook-formula for enumerating border-strip tableaux,
see \cite{Fomin1997}, but less study has been devoted to 
enumerating border-strip decompositions. 
Even determining if a region can be tiled by $n$-ribbons is non-trivial,
see \cite{Pak2000Ribbon}.

\medskip 

We introduce a family of shapes (called \emph{simple diagrams}) which have 
nice properties with respect to enumeration.
These are parametrized by a binary word, and the size of the ribbons which are used to tile the region.
In particular, we show that certain normalized enumerations grow as a polynomial in $n$ (the size of the ribbons)
thus reducing specific enumerations to a finite computation.

\subsection{Overview of results}

We show that border-strip tableaux and border-strip decompositions of simple diagrams
are in bijection with certain classes of permutations,
see \cref{Injection_to_S_m} and \cref{cor:SimpleBSDCharacterization}.
This allows us to study a certain $q$-analogue of 
border-strip decompositions, which generalize the classical inversion-statistic on permutations.
For example, in \cref{cor:totalNumberOfBorderStripDecomps},
we give the formula
\[
 \sum_{\wvec \in  \{r,c\}^k} \sum_{T \in \BSD(\wvec,n)} q^{\inv T} = [n+1]^k_q [n]_q!
\]
where the first sum is over all binary words of length $k$ (defining a simple diagram),
and $\BSD(\wvec,n)$ is the set of border-strip decompositions with strips of size $n$,
and shape determined by $(\wvec,n)$.
\medskip
In \cref{Polynomiality}, we give an efficient way to compute the number of 
border-strip decompositions of simple diagrams, as a function of $n$ --- the strip size.
This allows us to prove an inequality, showing that ``straighter'' 
simple shapes admit a larger number of border-strip decompositions, see \cref{thm:straighterInequality}
The maximum is attained for rectangles.
In contrast, by \cref{cor:countingBST} we know that these shapes 
admit the same number of \emph{border-strip tableaux} whenever $n\geq k$.
\medskip 

Finally, we give a new interpretation of \cite[\texttt{A115047}]{OEIS} in the OEIS.
We show that these numbers count the number of ways to tile a $2n \times n$-rectangle with strips of size $n$,
which gives a new simple combinatorial interpretation of certain Weil--Petersson volumes.
We cannot give an intuitive explanation for this curious 
connection, and it invites for further research.

\section{Preliminaries}

We first need to recall some general definitions ---
for a thorough background, see \cite{StanleyEC2}.

A \emph{tableau} of shape $\lambda$ and type $\mu$ is a filling of the Young diagram $\lambda$,
such that there are exactly $\mu_i$ boxes filled with $i$, for $i=1,\dotsc,\length(\mu)$.
A \emph{border-strip} (or simply \emph{strip}) of a diagram is a subset of boxes that form a connected skew shape,
and contains no $2\times2$ subdiagram.
A \emph{border-strip tableau}\footnote{Also known as rim-hook tableau} 
is a tableau such that rows and columns are weakly increasing,
and for all $i$, the boxes filled with the number $i$, form a border-strip.
We let $\BST(\lambda,\mu)$ denote the set of border-strip tableaux of shape $\lambda$ and type $\mu$.

A \emph{border-strip decomposition} of shape $\lambda$ and type $\mu$
is a partition of $\lambda$ into border-strips where the border-strip sizes are determined by the $\mu_i$,
and the set of such decompositions is denoted $\BSD(\lambda,\mu)$.
Hence, each border-strip tableau defines a border-strip decomposition.
Finally, the definition of $\BSD(\lambda,\mu)$ extends in the natural manner the case when $\lambda$ is a skew shape.

\begin{example}
The following tableau $T$ is an element in $\BST(\lambda,\mu)$
with $\lambda=(5,5,4,3,3,3)$ and $\mu=(5,4,3,4,3,2,2)$.
To the right, we show the corresponding border-strip decomposition with the strips indicated by the colors.
\begin{align}\label{eq:borderTableauExample}
\ytableausetup{centertableaux,boxsize=1.0em}
\begin{ytableau}
1 & 1 & 1 & 1 & 4 \\
1 & 2 & 4 & 4 & 4 \\
2 & 2 & 5 & 5 \\
2 & 3 & 5 \\
3 & 3 & 7 \\
6 & 6 & 7
\end{ytableau}
\qquad  
\begin{ytableau}
*(blue) & *(blue) & *(blue) & *(blue)  & *(red)  \\
*(blue)   & *(gray)  & *(red)   & *(red) & *(red) \\
*(gray)  & *(gray)  & *(yellow) & *(yellow)  \\
*(gray) & *(green) & *(yellow) \\
*(green)  & *(green) & *(purple)  \\
*(brown) & *(brown)  & *(purple) 
\end{ytableau}
\end{align}
\end{example}
It is clear that the number of elements in $\BST(\lambda,\mu)$ depend on the order of the entries in $\mu$,
but this is not the case for $\BSD(\lambda,\mu)$.
In particular, $\BST(\lambda,\mu)$ might be empty, while $\BSD(\lambda,\mu)$ is not.

\medskip 

Recall that the \emph{content}, $c(\square)$, of a box is defined as the difference $j-i$
of column-index minus row-index of the box.
From the definition of border-strips, it is straightforward to show that
the boxes in a border-strip $B$ all have different content, and these numbers form the \emph{content-interval} $a,a+1,\dotsc,b$ with no gaps.
We can thus define the \emph{head}, $\head{B}$ of a border-strip is the box with maximal content,
and its \emph{tail}, $\tail{B}$, which is the box with minimal content. 
In \eqref{eq:headTail}, the head and tail boxes have been marked.
\begin{align}\label{eq:headTail}
\begin{ytableau}
\none\, & \none & \none & \none & H \\
\none & \none & \, & \, & \, \\
\none & \none & \,  \\
T & \, & \, \\
\end{ytableau}
\end{align}

\section{Enumeration of border-strip decompositions}

In this section, we introduce a natural family of diagram shapes 
which have particularly nice properties. 

We first describe a bijection from border-strip decompositions of such shapes
to certain permutations. Using this bijection,
we are able to give several $q$-refinements of enumerations of border-strip decompositions.
In particular, this includes the classical $q$-analogue of permutations in $\symS_n$ given 
by Mahonian statistics.

\begin{definition}
A \emph{simple diagram} is parametrized by two parameters, a word $\wvec$ with entries in $\{r,c\}$,
and a natural number $n$. 

The family of simple diagrams are constructed recursively as follows:
\begin{itemize}
 \item If $\wvec = \emptyset$, then $(\wvec,n)$ is the $n\times n$-square.
 \item The diagram $(c\wvec,n)$ is obtained from $(\wvec,n)$ by adding 
 an additional column of size $n$ on the left, such that the bottom-most 
 square of the new column is in the bottommost row of $(\wvec,n)$.
 
  \item The diagram $(r\wvec,n)$ is obtained from $(\wvec,n)$ by adding 
 an additional row of size $n$ on the bottom, such that the left-most 
 square of the new row is in the leftmost column of $(\wvec,n)$.
\end{itemize}
\end{definition}

We let $\BSD(\wvec,n)$ denote the set of border-strip decompositions of $(\wvec, n)$,
and $\BST(\wvec,n)$ denotes the set of border-strip tableaux of $(\wvec, n)$,
in both cases with strips of size $n$.

For a word $\wvec$, we define $C_{\wvec}$ the total number of $c$'s in $\wvec$, $R_{\wvec}$ the total 
number of $r$'s in $\wvec$.
Furthermore, let $\hor(\wvec) \coloneqq C_{\wvec}-R_{\wvec}$. 
Intuitively, $\hor(\wvec)$ measures how ``horizontal'' the diagram is.

\begin{example}
The simple diagram determined by $(rcrcc,2)$ is the following shape:
\begin{align}
\ytableausetup{centertableaux,boxsize=1.0em}
\begin{ytableau}
\none & & & &\\
& & & &\\
&  & \\
& 
\end{ytableau}
\qquad  
\end{align} 
Below we can see how $(rcrcc,2)$ is constructed from the $2\times 2$ 
square by adding successively the blue, red, green, yellow and gray boxes to a $2\times2$ square.
\begin{align}
\ytableausetup{centertableaux,boxsize=1.0em}
\begin{ytableau}
\none &*(red) c&*(blue) c& &\\
*(yellow) c&*(red) c &*(blue) c& &\\
*(yellow) c&*(green) r & *(green) r\\
*(gray) r& *(gray) r
\end{ytableau}
\qquad  
\end{align}
We have $C_{rcrcc}=3$, $R_{rcrcc}=1$ and $\hor(rcrcc)=3-2=1$.
\end{example}

\begin{definition}\label{def:comparable}
In a fixed border-strip decomposition, a border-strip $B_a$ is \emph{above} a border-strip $B_b$ if there is 
a path from $B_a$ to $B_b$ going only down or right. In this case $B_b$ is \emph{below} $B_a$.
\end{definition}
\begin{definition}
A border-strip $B_a$ is \emph{inner} to a border-strip $B_b$ if there exists a sequence $B_a=B_1,B_2,\dotsc,B_k=B_b$ such that for all $i$ $B_i$ is above $B_{i+1}$. This means the relation \emph{inner} is the transitive closure of the relation \emph{above}.

If $B_a$ is inner to $B_b$, $B_b$ is \emph{outer} to $B_a$.

Two border strips $B_a$ and $B_b$ are \emph{comparable}, if $B_a$ is inner or outer to $B_b$.
\end{definition}

\begin{remark}
If $B_1$ is above $B_2$, it implies $B_1$ must contain a 
smaller number than $B_2$ in any border-strip tableau, 
thus the existence of a BST for any BSD implies the transitive closure is well-defined.

Also, $B_1$ is inner to $B_2$ if and only if it contains a smaller 
number in every BST with the border-strip decomposition.
We do not use this property, but it follows from the proof of \cref{CountingBSD} below.
\end{remark}

\begin{example}
Here is an example in $\BSD(ccrcc,3)$:
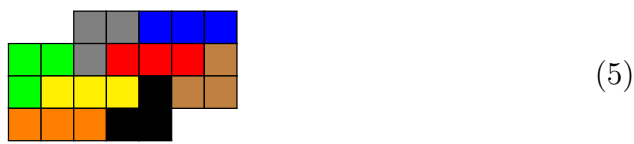
\begin{align}
\ytableausetup{centertableaux,boxsize=1.0em}
\begin{ytableau}
\none & \none &*(gray) & *(gray)& *(blue) & *(blue) & *(blue)\\
*(green)& *(green)&*(gray) & *(red) & *(red) & *(red) &*(brown)\\
*(green) & *(yellow)& *(yellow) & *(yellow) & *(black)&*(brown)&*(brown)\\
*(orange)&*(orange)&*(orange)&*(black)&*(black)
\end{ytableau}
\qquad  
\end{align}
In this case the blue strip is above the red strip, and the red strip is above the yellow strip, which means the blue strip is inner to the yellow strip, and the blue and yellow strip are comparable. But the blue strip is neither above nor below the yellow strip.
\end{example}

\begin{definition}
Two border strips $B_1$ and $B_2$ in a decomposition 
form an \emph{inversion} if the following three conditions are fulfilled:
\begin{itemize}
 \item The content-sequences of $B_1$ and $B_2$ have a non-empty intersection,
 \item $B_1$ is inner to $B_2$, and
 \item $\head{B_1}>\head{B_2}$.
\end{itemize}
\end{definition}
We prove in \cref{cor:MahonianEnumeration} that this definition 
generalizes the notion of inversions in $\symS_n$ in a natural manner.

\begin{definition}
For a word $\wvec$ of length $k$, we number the diagonals of the simple diagram $(\wvec, n)$ from $n+k$ to $1$, 
starting in the top right corner, as shown in the example below for $(crrc,3)$:
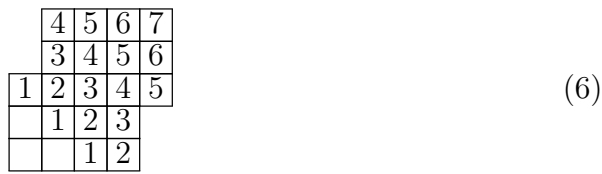
\begin{align}
\ytableausetup{centertableaux,boxsize=1.0em}
\begin{ytableau}
\none 3&4 &5 &6 &7\\
\none2 1& 3& 4&5 &6\\
1& 2& 3&4 &5\\
&1 & 2&3\\
& &1 &2
\end{ytableau}
\qquad  
\end{align}
\end{definition}

\begin{lemma}
Let $\wvec$ be a word of length $k$. Then for any decomposition in $\BSD(\wvec,n)$, there is a unique head in 
each diagonal from 1 to $n+k$, and the position of 
the heads uniquely determines the border-strip decomposition.
\end{lemma}
\begin{proof}
We will show that the position of the heads uniquely determines the decomposition,
by processing the diagonals one by one and iteratively prolonging the strips, 
starting from diagonal $n+k$.

The only way to cover the single box in diagonal $n+k$ is for it to be a head.

For diagonal $i$ with $k<i<n+k$ we have one box more in diagonal $i$ than in 
diagonal $i+1$, and all strips we already started have less than $n$ squares, and must continue, 
therefore there is exactly one head in diagonal $i$. 
Furthermore, the position of the head $H$ in diagonal $i$ determines the continuation of 
the strips started, as shown in this figure:
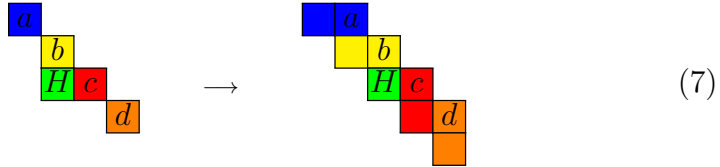
\begin{align}
\begin{ytableau}
 *(blue) a\\
\none & *(yellow) b\\
\none & *(green) H & *(red) c\\
\none & \none & \none &*(orange) d\\
\none & \none & \none & \none
\end{ytableau}
\qquad
\substack{ \\ \longrightarrow  }
\qquad
\begin{ytableau}
*(blue) & *(blue) a\\
\none  & *(yellow)  & *(yellow) b\\
\none & \none & *(green) H & *(red) c\\
\none & \none & \none & *(red)&*(orange) d\\
\none & \none & \none & \none  &*(orange)
\end{ytableau}
\end{align}\\
For $i\leq k$, there is exactly one strip ending in diagonal $i+1$, and 
diagonals $i$ and $i+1$ have the same size, therefore there must be exactly 
one head in diagonal $i$. Once we placed the head, there are $n-1$ boxes left 
in diagonal $i$, and $n-1$ strips must have a box in diagonal $i$. 
As strips cannot cross each other, this gives at most one solution.

Similarly, for the diagonals below diagonal 1, the size of the diagonals 
decreases by 1 each step, and the number of strips too, so there cannot be any heads 
below diagonal 1, and there is a unique way to extend the border-strip decomposition.
\end{proof}

\begin{definition}
Given a border-strip decomposition of a simple diagram,
the unique strip with head in diagonal $i$ is referred to as \emph{strip $i$}.
\end{definition}

\begin{proposition}\label{comparable}
Let $(\wvec, n)$ be a simple diagram. 
Then for any decomposition in $\BSD(\wvec,n)$, if $|i-j|\leq n$, then strip $i$ and $j$ are comparable.
\end{proposition}
\begin{proof}
Without loss of generality, $i>j$. Then the tail of $i$ is at most one diagonal 
higher than the head of $j$. As we can cover two consecutive diagonals with a 
path going only right and down, two elements that are at most one diagonal apart are comparable.
\end{proof}

\medskip 

We noticed that the positions of the heads of the strips uniquely determine the
border-strip decomposition. The next definition and proposition
encodes the placements of the heads as a permutation with certain restrictions,
giving an alternative description of border-strip tableaux of simple shapes.

Further down, we add more restrictions, so that the resulting set of permutations 
are in bijection with border-strip decompositions.

\begin{definition}
We define $\psi:\BST(\wvec,n)\rightarrow \symS_{n+k}$ by $\psi(T)=\sigma$ 
such that if the unique head in diagonal $i$ is numbered $j$ then $\sigma(j)=i$.

We let $\BSP(\wvec,n) \subseteq \symS_{n+k}$ denote the image of $\BST(\wvec,n)$ under $\psi$.
\end{definition}

\begin{proposition}\label{injectivity}
The map $\psi$ is injective.
\end{proposition}
\begin{proof}
A permutation defines the value of the heads in each diagonal, 
and thus the value of all the boxes in each diagonal. 
As they have to be in increasing order to form a border-strip tableau, 
there is a unique way to do this. 

Note that not every permutation give rise to a valid border-strip tableau,
see \cref{Injection_to_S_m} below.
\end{proof}
\begin{example}
Here is an example $T \in$ $\BST(ccc,n)$:
\begin{align}
T=
\ytableausetup{centertableaux,boxsize=1.0em}
\begin{ytableau}
1 & 1 & \mathbf{1} & 3 & \mathbf{3} & \mathbf{4} \\
2 & 2 & \mathbf{2} & 3 & 4 & 4 \\
5 & 5 & \mathbf{5} & 6 & 6 & \mathbf{6}
\end{ytableau}
\qquad  
\text{ with }
\qquad
\psi(T) = [3,2,5,6,1,4].
\end{align}
The strip labeled $1$ has its head in diagonal $3$, thus $\psi(T)(1)=3$, the strip 
labeled $2$ has its head in diagonal $2$, thus $\psi(T)(2)=2$ and so on.
\end{example}

\begin{proposition}\label{Injection_to_S_m}
Let $\wvec = (w_1,\dotsc,w_k)$ be a word of length $k$. 
A permutation $\sigma \in \symS_{n+k}$ is in $\BSP(\wvec,n)$ 
if and only if for all $i$ with $1\leq i\leq k$ 
we have:
\begin{itemize}
 \item 
$\sigma^{-1}(i) <\sigma^{-1}(n+i)$ whenever $w_i=c$, and
 \item 
$\sigma^{-1}(i) >\sigma^{-1}(n+i) $ whenever $w_i=r$.
\end{itemize}
\end{proposition}
\begin{proof} We construct the tableau from the last diagonal to the first one. 
For any $i$, the unique head in diagonal $i$ must be filled with 
number $\sigma^{-1}(i)$. If $k<i\leq n+k$, diagonal $i$ has one element 
more than diagonal $i+1$, and it is always possible to extend a BSD. 
If $1\leq i \leq k$, we have to look at $w_i$.
If $w_i=c$, diagonals $i$ and $i+1$ are as follows:
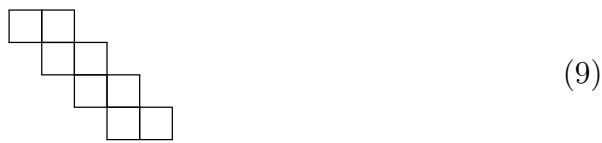
\begin{align}
\ytableausetup{centertableaux,boxsize=1.0em}
\begin{ytableau}
 \empty& \\
 \none & &\\
 \none & \none & &\\
 \none & \none & \none & &
\end{ytableau}
\qquad  
\end{align}
We observe the new strip must be added above the strip starting in 
diagonal $n+i$ (ending in diagonal $i+1$), 
which means it has to be a smaller number, \emph{i.e.} $\sigma^{-1}(i)<\sigma^{-1}(n+i)$.
If $w_i=r$, diagonals $i$ and $i+1$ must be as follows:
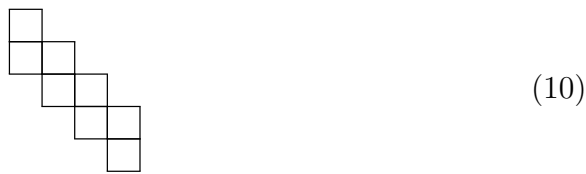
\begin{align}
\ytableausetup{centertableaux,boxsize=1.0em}
\begin{ytableau}
 \empty\\
  &\\
  \none & &\\
  \none & \none & &\\
  \none & \none & \none &
\end{ytableau}
\qquad  
\end{align}
and the new strip must be below strip $n+i$, 
and it has to be filled with a larger number, \emph{i.e.} $\sigma^{-1}(i)>\sigma^{-1}(n+i)$.
\end{proof}

\begin{corollary}\label{cor:countingBST}
For a word $\wvec$ of length $k$, with $k\leq n$, we have 
\[
|\BST(\wvec,n)|=(n+k)!/2^k
\]
This means the number of border-strip tableaux only depends on the length of the word for $n\geq k$.
In contrast, this count is word dependent for $n<k$.
\end{corollary}
\begin{proof}
From \cref{injectivity} we know $\psi$ is injective, thus $|\BST(\wvec,n)|=|\BSP(\wvec,n)|$ 
From the conditions in \cref{Injection_to_S_m} we know the 
relative order on all pairs $(i,n+i)$. 
As $n\geq k$, no such entry belongs to two such pairs, and thus $|\BSP(\wvec,n)|=(n+k)!/2^k$.
\end{proof}

\begin{corollary}
For every permutation $\sigma\in \symS_{n+k}$ there is exactly 
one word $\wvec$ of length $k$ such that $\sigma\in \BSP(\wvec,n)$.
In particular, $\psi$ is a bijection 
between $\lbrace \BST(\wvec,n):\wvec\in\lbrace r,c\rbrace^k\rbrace$ and $\symS_{n+k}$ and
\[
\sum_{\wvec\in\lbrace r,c\rbrace^k}|\BST(\wvec,n)|=(n+k)!.
\]
\end{corollary}
\begin{proof}
From \cref{injectivity} we know $\psi$ restricted to one word is injective. 
From \cref{Injection_to_S_m} we deduce two different words cannot give the same 
permutation so $\psi$ is injective over the set of all words of length $k$. 
On the other hand, 
for any permutation $\sigma\in\symS_{n+k}$ 
there is always one word $\wvec\in \lbrace r,c\rbrace^k$ such that $\sigma\in \BSP(\wvec,n)$,
as we can recover the word from the pairs $(i,n+i)$.
Thus $\psi$ is also surjective.
\end{proof}

\begin{definition}
If $\sigma(i)-k>\sigma(i+1)$, $i$ is called a $k$-\emph{descent} of $\sigma$.
Let $\DES_k(\sigma)$ denote the set of $k$-descents of $\sigma$,
and let $\des_k(\sigma)$ be the number of such $k$-descents.
\end{definition}

\begin{example}
Let $\sigma=[2,4,\mathbf{10},5,6,3,\mathbf{8},1,7,9]$, 
then the $3$-descents of $\sigma$ are $3$ and $7$, marked in bold.
\end{example}

Let $s_1,\dotsc,s_{n-1}$ denote the simple transpositions in $\symS_n$. 
\begin{proposition}\label{CountingBSD}
Let $\wvec$ be a word of length $k$, $\sigma \in \BSP(\wvec,n)$ and $i\in \DES_n(\sigma)$, 
then the border-strip tableaux $\psi^{-1}(s_i\sigma)$ and $\psi^{-1}(\sigma)$ 
give rise to the same border-strip decomposition.
Moreover, 
\[
 |\BSD(\wvec,n)| = |\{ \sigma \in \BSP(\wvec,n) : \des_n(\sigma)=0 \}|
\]
that is, the number of elements in $\BSD(\wvec,n)$ is the number of permutations in $\BSP(\wvec,n)$ without $n$-descent.
\end{proposition}
\begin{proof}
Let $\tau \coloneqq s_i \sigma$ and $T_\sigma$, $T_\tau$ be the corresponding border-strip tableaux. 
First we show that $\tau \in \BSP(\wvec,n$).
The only places where $\tau^{-1}$ differs from $\sigma^{-1}$ are $\tau(i)$ and $\tau(i+1)$. 
As $i$ is an $n$-descent, \cref{Injection_to_S_m} does not give any condition on their order. 
Suppose $j \in [k]$. Then at most one from 
$\sigma^{-1}(j) \text{ and }
\sigma^{-1}(n+j)
$
is different for $\tau^{-1}$ and
the quantities
\[
 \tau^{-1}(j)-\tau^{-1}(n+j) \text{ and } \sigma^{-1}(j)-\sigma^{-1}(n+j)
\]
are either the same or differ by $1$, so they can never have opposite signs.
Since $\sigma \in \BSP(\wvec,n)$, the conditions in \cref{Injection_to_S_m}
are still fulfilled for $\tau$ and we have that $\tau \in \BSP(\wvec,n)$.
\medskip 

\emph{It remains to show that $T_\tau$ and $T_\sigma$ 
have the same border-strip decomposition.}
The only strips which have a different number in $T_\sigma$ and $T_\tau$
are strip $\tau(i)$ and strip $\tau(i+1)$ and the new numbers differ by $\pm 1$.
Therefore, the only pair that has a different relative
ordering under $\tau$ than under $\sigma$ is the pair $(\tau(i),\tau(i+1))$.
However, since $i$ is an $n$-descent, it does not affect the construction
in the proof of \cref{Injection_to_S_m}, and as $\psi$ is injective, this implies 
that $T_\sigma$ and $T_\tau$ have the same BSD.

\medskip 

For the second statement, we will prove that there 
is exactly one permutation without any $n$-descent in $\BSP(\wvec,n)$
for a fixed border-strip decomposition.

\medskip
We claim that if there are two strips, $x$ and $y$, such that the three following conditions hold:
\begin{enumerate}
\item the strips $x$ and $y$ are not comparable in the sense of \cref{def:comparable}
\item $x>y$ 
\item $\sigma^{-1}(x)<\sigma^{-1}(y)$
\end{enumerate}
then $\sigma$ has an $n$-descent.

We consider the sequence $\sigma^{-1}(x)=a_1,a_1+1=a_2,\dotsc,a_m=\sigma^{-1}(y)$ and $i$ 
such that $\sigma(a_i)-\sigma(a_{i+1})$ is maximal.
If $\sigma(a_i)-\sigma(a_{i+1})\leq n$, then we can find a 
subsequence $\sigma^{-1}(x)=a_{i_1},a_{i_2},\dotsc,a_{i_s}=\sigma^{-1}(y)$ such that 
for all $j$ we have $|\sigma(a_{i_j})-\sigma(a_{i_j+1})|\leq n$.
But then \cref{comparable} implies $\sigma(a_{ij})$ and $\sigma(a_{ij+1})$ are comparable, 
and by transitivity, $x$ and $y$ are comparable, which contradicts our assumption.

This implies to avoid an $n$-descent, we must fix the relative order 
of all non-comparable pairs, but the relative order of comparable 
pairs is always fixed, 
which means there is at most one permutation without $n$-descent for a given decomposition.

On the other hand, we can always find such a permutation,
by starting from a permutation in $\BSP(\wvec,n)$ and repeatedly 
remove $n$-descents until a permutation without $n$-descents is obtained.
\end{proof}

\begin{corollary}\label{cor:SimpleBSDCharacterization}
Let $\wvec \in \{r,c\}^k$. The set of border-strip decompositions of the simple diagram $(\wvec,n)$
is in bijection with the set of permutations in $\symS_{n+k}$ such that
for each $i \in [k]$,
\begin{itemize}
 \item $w_i = c \quad \Longrightarrow \quad \sigma^{-1}(i) < \sigma^{-1}(n+i)$,
 \item $w_i = r \quad \Longrightarrow \quad \sigma^{-1}(i) > \sigma^{-1}(n+i)$ and
 \item $\sigma(j)-\sigma(j+1)\leq n$ for all $j\in[n+k-1]$.
\end{itemize}
\end{corollary}

\bigskip 

\begin{definition}
For a word $\wvec \in \{r,c\}^k$ let 
\begin{equation}
 \hat{f}_{\wvec}(n):=|\BSD(\wvec,n)|\frac{(2k)!}{(n-k)!}.
\end{equation}
\end{definition}

\begin{proposition}\label{Polynomiality}
Whenever $n>2k-1$, the function $\hat{f}_{\wvec}(n)$ is equal to
\begin{align}\label{eq:fFormula}
f_{\wvec}(n) = \sum_{\tau \in \BSP(\wvec,k)}(n+k-\des_k(\tau))_{2k}.
\end{align}
As a consequence, $\hat{f}_{\wvec}(n)$ is a polynomial in $n$ of 
degree $2k$ with integer coefficients when restricted to 
values $n>2k-1$. Moreover, $f_{\wvec}(n)$ is divisible by the falling factorial $(n+1)_{k+1}$.
\end{proposition}
\begin{proof}
% From \cref{cor:SimpleBSDCharacterization} we know we can count the number of 
% permutations in $\symS_{n+k}$ with certain conditions. 
Interpreting permutations in $\symS_{n+k}$ as sequences of $n+k$ numbers, 
we note that the first two conditions in \cref{cor:SimpleBSDCharacterization} only 
apply to the relative order of the first and last $k$ elements. 
%
% 
% for each $i \in [k]$,
% \begin{itemize}
%  \item $w_i = c \quad \Longrightarrow \quad \sigma^{-1}(i) < \sigma^{-1}(n+i)$,
%  \item $w_i = r \quad \Longrightarrow \quad \sigma^{-1}(i) > \sigma^{-1}(n+i)$ and
%  \item $\sigma(j)-\sigma(j+1)\leq n$ for all $j\in[n+k-1]$.
% \end{itemize}

Thus, in order to construct a permutation $\sigma$ in $\symS_{n+k}$ fulfilling the three conditions in 
\cref{cor:SimpleBSDCharacterization}, we proceed in three initial steps:

\begin{enumerate}
\item Choose an ordering of the entries $1,2,\dotsc,k,n+1,n+2,\dotsc,n+k$.
\item Choose the positions of the entries $1,2,\dotsc,k,n+1,n+2,\dotsc,n+k$.
\item Choose an ordering of the entries $k+1,\dotsc,n$.
\end{enumerate}
Not all choices here will fulfill the conditions in 
\cref{cor:SimpleBSDCharacterization}, we shall see below which ones are valid.
For a choice in the first step, two things might happen:
\begin{enumerate}
\item[a)] There is some pair $(i,i+k)$ in the wrong order --- violating one of the first two conditions. 
In this case we do not have a BST, and thus no BSD corresponding to this choice.
\item[b)] All pairs $(i,i+k)$ have the correct order.
In this case, the ordering of the entries
\[
1,2,\dotsc,k,n+1,n+2,\dotsc,n+k
\]
fulfill the conditions (after standardization) of being a permutation $\tau$ in $\BSP(\wvec,k)$.
% By extending $\sigma$ we can construct a permutation in $\symS_{n+k}$ 
% fulfilling the conditions of \cref{cor:SimpleBSDCharacterization} 
% with the chosen order on the first and last $k$ elements.
\end{enumerate}
Now we need to ensure that there are no $n$-descents in the final permutation. 
If there are no $k$-descents in $\tau$ (from step $b$ above), this is always the case.
Otherwise, we need to insert another number after every $k$-descent of $\tau$. 
This means we only have $\binom{n+k-\des_k(\tau)}{2k}$ valid choices in step (2). 
The last step always has $(n-k)!$ valid choices as the order on $k+1,\dotsc,n$ does not matter. 
It follows that $f_{\wvec}(n)$ is given by 
\begin{align*}
f_{\wvec}(n) & =\frac{(2k)!}{(n-k)!} \sum_{\tau \in \BSP(\wvec,k)}  \binom{n+k-\des_k(\tau)}{2k} (n-k)! \\
  & =\sum_{\tau \in \BSP(\wvec,k)}(n+k-\des_k(\tau))_{2k}.
\end{align*}
This function is obviously a polynomial of degree $2k$. 
Furthermore, since $\des_k(\tau)$ is between $0$ and $k-1$
it follows that $(n+k-\des_k(\tau))_{2k}$ is divisible by $(n+1)_{k+1}$.
\end{proof}

\begin{corollary}
We have the enumeration
\[
|\BSD(rc,n)|=(n+1)!(3n+2)/12 \text{ whenever } n\geq2.
\]
\end{corollary}
\begin{proof}
Using \cref{Polynomiality}, we know that $|\BSD(rc,n)|$ can 
be expressed as $(n-2)!\hat{f}_{rc}(n)/4!$.
Since we know that $\hat{f}_{rc}(n)$ is a polynomial in $n$ for $n \geq 4$, 
it suffices to verify the formula for the first few values of $n$.
\end{proof}
The sequence $a_{n+1} = (n+1)!(3n+2)/12$ appear as \cite[\texttt{A227404}]{OEIS}, 
where $a_n$ count the total number of inversions in all permutations in $\symS_n$
consisting of a single cycle. For example, the permutations $(123)$ and $(132)$
have four inversions in total, giving $a_3=4$.

\begin{lemma}
Let $\sigma \in \BSP(\wvec,n)$, with $T_\sigma$ being the corresponding border-strip decomposition.
Then the strips $i$ and $j$ in $T_\sigma$
with $i<j$ form an inversion if and only if $j-i<n$ and $\sigma^{-1}(i)>\sigma^{-1}(j)$.
\end{lemma}
\begin{proof}
If $j-i\geq n$ they do not have an element on the same diagonal, 
and by definition do not form an inversion. 
If $j-i<n$ they share an element on the same diagonal, 
and if $\sigma^{-1}(i)>\sigma^{-1}(j)$ strip $j$ is above strip $i$, and we have an inversion.
\end{proof}
Given a border-strip decomposition $T$, let $\inv(T)$ denote the total number of inversions in $T$. 
Furthermore, for $\sigma \in \symS_{n+k}$ let
\[
 \inv_n(\sigma) \coloneqq \{ (i,j) : 0<j-i<n \text{ and } \sigma^{-1}(i)>\sigma^{-1}(j) \}.
\]

The $q$-analogue of $\BSD(\wvec,n)$ is defined as
\begin{align}
 \sum_{T \in \BSD(\wvec,n)} q^{\inv(T)}
\end{align}
and by previous lemma we have that 
\begin{align}
 \sum_{T \in \BSD(\wvec,n)} q^{\inv(T)} = \sum_{\sigma \in \BSP(\wvec,n)} q^{\inv_n(\sigma)}.
\end{align}

\begin{corollary}\label{cor:MahonianEnumeration}
The $q$-analogue of the $n\times n$-square, $\BSD(\emptyset,n)$, satisfies the identity
\[
\sum_{T \in \BSD(\emptyset,n)} q^{\inv(T)} = [n]_q!.
\]
\end{corollary}
\begin{proof}
From \cref{Injection_to_S_m} we know all permutations in $\symS_n$ are 
in $\BSP(\emptyset,n)$, from \cref{CountingBSD}, 
we know all BST correspond to BSD, and from the previous 
result we deduce the $q$-analogue is given by $[n]_q!$
\end{proof}

\begin{corollary}
We have the following $q$-analogue for $\BSD(c,n)$:
\[
\sum_{T\in \BSD(c,n)} q^{\inv T} = [n-1]_q! \sum_{i=1}^n i q^{i-1}.
\]
\end{corollary}
\begin{proof}
We get a permutation corresponding to a decomposition by placing $1$ and $n+1$ 
(\emph{i.e}. choose $\sigma^{-1}(1)$ and $\sigma^{-1}(n+1)$), and then choose the 
order of $2,\dotsc,n$. This choice gives $[n-1]_q!$, and the 
possible positions of $1$ and $n+1$ gives $\sum_{i=1}^n i q^{i-1}$, as $1$ has 
to be before $n+1$ for it to be a BST. 
Note that there cannot be any $n$-descents and therefore the number of border-strip tableaux 
is equal to the number of decompositions.
\end{proof}

\begin{proposition}
If $\wvec$ is a word of a simple diagram, then 
\[
|\BSD(c\wvec,n)|+|\BSD(r\wvec,n)|=(n+1)|\BSD(\wvec)|.
\]
Furthermore, this relation extends to the following $q$-analogue:
\[
\sum_{T\in \BSD(c\wvec,n)} q^{\inv T} + \sum_{T\in \BSD(r\wvec,n)} q^{\inv T} 
= [n+1]_q \sum_{T\in \BSD(\wvec,n)} q^{\inv T}
\]
\end{proposition}
\begin{proof}
If we fix the positions of the heads in $(\wvec ,n)$, the new head in $(c\wvec, n)$ must 
be above the strip it replaces, where as in $(r\wvec, n)$ it must be below. 
Together, this gives $n+1$ possibilities to complete a BSD of $(\wvec, n)$. 
If, in $(c\wvec, n)$ or in $(r\wvec ,n)$, we place the new head in position $i$ of the diagonal, 
the new strip forms an inversion with all $i-1$ strips above it, thus the $q$-analogue.
\end{proof}

\begin{corollary}\label{cor:totalNumberOfBorderStripDecomps}
We can count the total number of border-strip decompositions for all words of length $k$, more precisely:
\[
\sum_{\wvec\in\lbrace r,c\rbrace^k}|\BSD(\wvec,n)|=(n+1)^k n!
\]
and this relation extends to the $q$-analogue:
\[
 \sum_{\wvec \in \{r,c\}^k } \sum_{T \in \BSD(\wvec,n)} q^{\inv T}=[n+1]_q^k[n]_q!
\]
\end{corollary}
\begin{proof}
It suffices to show the $q$-analogue, by taking $q=1$ we obtain the enumeration. We proceed by induction.

The base case, $k=0$, is given by \cref{cor:MahonianEnumeration}. The previous result gives the induction step:
\[\sum_{\wvec \in \{r,c\}^k } \sum_{T \in \BSD(\wvec,n)} q^{\inv T}=\]
\[
\sum_{\wvec \in \{r,c\}^{k-1} } \sum_{T \in \BSD(r\wvec,n)} q^{\inv T}+\sum_{\wvec \in \{r,c\}^{k-1} } \sum_{T \in \BSD(c\wvec,n)} q^{\inv T}=\]
\[
[n+1]_q\sum_{\wvec \in \{r,c\}^{k-1} } \sum_{T \in \BSD(\wvec,n)} q^{\inv T}
\]
\end{proof}

If we let $n=k-1$, we note that the sequence $a(n) = (n+1)^{n-1} n!$ is \texttt{A066319}.
This sequence also show up in \cite[Thm. 5.4]{Weist2012}.
Let $K_{n,n+1}$ be the complete bipartite graph with $n$ sources and $n+1$ sinks.
Then there are $a(n)$ spanning trees such that every source has exactly $2$ incident edges.
This is related to computing the Euler characteristic of certain moduli spaces,
see \cite{Weist2012} for details. 
This connection is quite interesting, as it is perhaps related to what we discuss in \cref{seq:wpConnection} below.

Recall the definition of $\hor(\wvec)$ as the difference between 
the number of occurrences of $c$ and $r$ in $\wvec$.
The following theorem shows that ``straighter'' shapes admits a larger number of decompositions,
in a precise sense:
\begin{theorem}\label{thm:straighterInequality}
If $\vvec$ and $\wvec$ are words of length $k$ and $|\hor(\vvec)|<|\hor(\wvec)|$, 
then 
\[
 |\BSD(\vvec,n)| > |\BSD(\wvec,n)| \text{ for $n$ sufficiently large.}
\]
In fact, 
\[
 \frac{|\BSD(\vvec,n)| - |\BSD(\wvec,n)|}{(n-k)!} = O(n^{2k-1}).
\]
\end{theorem}
\begin{proof}
Recall from from \cref{Polynomiality} that 
\[
f_\vvec(n)=\sum_{\sigma \in \BSP(\vvec,k)}(n+k-\des_k(\sigma))_{2k}.
\]

From \cref{cor:countingBST}, we know that $|\BSP(\vvec,k)| = (2k)!/2^{2k}$.
It then follows that
\[
f_\vvec(n)=\frac{(2k)!}{2^k} n^{2k} + \alpha n^{2k-1} + \text{l.o.t} \text{ and  } 
f_\wvec(n)=\frac{(2k)!}{2^k} n^{2k} + \beta n^{2k-1} + \text{l.o.t}.
\]
Our goal is to prove that $\alpha<\beta$.

For a fixed permutation $\sigma \in \BSP(\vvec,k)$, its contribution to $\alpha$ is given by
\[
\sum_{i=0}^{2k-1} (k-\des_k(\sigma)- i ) = 2k^2-k(2k-1)-2k\des_k(\sigma).
\]
Hence,
\[
\alpha = k|\BSP(\vvec,k)| - 2k \sum_{\sigma \in \BSP(\vvec,k)} \des_k(\sigma).
\]
As $|\BSP(\vvec,k)|$ does not depend on $\vvec$, the only part depending on $\vvec$ is 
\[
J_\vvec:=\sum_{\sigma \in \BSP(\vvec,k)} \des_k(\sigma), 
\]
and it suffices to prove $J_\vvec$ is strictly smaller for a straighter word.

To do this, we count the number of permutations where $b+k$ is a $k-$descent with $a$, for $1\leq a<b\leq k$ 
fixed (\emph{i.e.} we have $\dotsc,b+k,a,\dotsc$ in the permutation). To create such a permutation, we can choose the 
order of all elements different from $a,b,a+k,b+k$ in any way respecting the orders 
of pairs $\sigma(i),\sigma(i+k)$, which gives $(2k-4)!/2^{k-2}$ choices. 
Then we must choose the order of the three blocks $a+k,(b+k)a,b$. If $a+k$ and $b$ 
are on the same side of $(b+k)a$, this gives two possibilities, otherwise there is 
only one way. We observe $a+k$ and $b$ are 
on the same side if and only if $\vvec_a \neq \vvec_b$. 
Finally, we can chose the position of the three blocks $a+k,b,(b+k)a$, which gives $\binom{2k}{3}$ choices. 
So the number of permutations where $b+k$ is a $k$-descent with $a$ is exactly 
\[
\begin{cases}
2\binom{2k}{3}(2k-4)!/2^{k-2} \text{ if $\vvec_a \neq \vvec_b$} \\
\phantom{2}\binom{2k}{3}(2k-4)!/2^{k-2} \text{ otherwise}.
\end{cases}
\]
Recall $C_\vvec$ is the number of $c$'s in $\vvec$, and $R_\vvec$ is the number of $r$'s in $\vvec$. The previous result implies
\[
J_\vvec=\binom{2k}{3}\frac{(2k-4)!}{2^{k-2}} \left[ 
\binom{C_\vvec}{2} + 2C_\vvec R_\vvec +  \binom{R_\vvec}{2}
\right]
\]
Since $R_\vvec=k-C_\vvec$, it follows that 
\[
\binom{C_\vvec}{2} + 2C_\vvec R_\vvec +  \binom{R_\vvec}{2} = \binom{k}{2} + C_\vvec R_\vvec
\]
which is increasing as $|\hor(\vvec)|$ decreases.
\end{proof}

\begin{conjecture}
The function $f_\wvec(n)$ uniquely define $\wvec$ up to isometry of the shape $\wvec$,
\emph{i.e.} up to exchanging $r$ and $c$ and reversing the word.
\end{conjecture}
Note that for fixed $k$, the polynomials (in $n$)
\[
 (n+k-i)_{2k} \text{ with $i=0,\dotsc,k-1$}
\]
are linearly independent: these span the same space as
\[
\left\{ \frac{(n+i)_{2k}}{(2k)!} \right\}_{i=1}^k \quad = \quad
\left\{ \binom{n+i}{2k} \right\}_{i=1}^k,
\]
and the latter collection of polynomials can be seen to be linearly independent.

As a consequence, given $f_\wvec(n)$, which is a sum over permutations in $\BSP(\wvec,k)$,
for any $i$ we can extract the number of permutations $\sigma \in \BSP(\wvec,k)$ with $\des_k(\sigma)=i$.
Hence, the conjecture is reduced to determining if the multi-set of $\des_k$-values 
of the elements in $\BSP(\wvec,k)$ uniquely determines $\wvec$ up to isometry.

In particular, if the number of terms without $k$-descents is different, 
the polynomial is also different, so we can formulate the stronger 
conjecture that $|\BSD(\wvec,k)|$ uniquely determines a word $\wvec$ of length $k$ up to isometry.

\section{A connection with the Weil--Petersson volume}\label{seq:wpConnection}

It follows from \cref{cor:SimpleBSDCharacterization} that 
the set $\BSD(2n \times n,n)$ is in bijection with
the set of permutations of $\{x_1,\dotsc,x_n,y_1,\dotsc,y_n\}$
such that $x_i$ appear before $y_i$ for all $i$,
and we do not have $\dotsc,x_i,y_j,\dotsc$ (consecutive), such that $i>j$.

\begin{lemma}[Adaptation of \cite{ZiqingXiang294267}]
The cardinality of $\BSD(2n\times n, n)$ is given by the formula
\begin{equation}\label{eq:sumAsGraphs}
|\BSD(2n\times n, n)| = \sum_{p \vdash n} (-1)^{|p - 1|} \frac{1}{m!} \binom{|p|}{p} \binom{|p + 1|}{p + 1}, 
\end{equation}
where $m=(m_1,m_2,\dotsc,m_k)$ and $m_i$ is the multiplicity of $i$ in $p$,
and we use the notation
$p\pm 1 \coloneqq (p_1 \pm 1,\dotsc,p_k \pm 1)$ and $|p| = p_1+\dotsb + p_k$.
Note that $\binom{|p|}{p}$  and $\binom{|p + 1|}{p + 1}$ denote a multinomial coefficients.
\end{lemma}
\begin{proof}
For a permutation $\sigma\in \symS_{2n}$ corresponding to a \emph{border-strip tableau}, 
let $\Gamma_\sigma$ be the graph on the vertex set $[n]$ with edge set
\[
\{ (\sigma(i)-n,\sigma(i+1)) : \sigma(i)-n>\sigma(i+1) \}.
\]
Let $G$ be the set of graphs obtained from such border-strip tableaux.
Let $E$ be $\binom{[n]}{2}$, that is, the set of all possible edges
on the vertex set $[n]$ and let $G(e_1,\dotsc,e_r) \subseteq G$
be the set of graphs that include the edges $\{e_1,\dotsc,e_r\}\subseteq E$.
By definition, elements in $\BST(2n \times n,n)$ are in bijection with $G(\emptyset)$,
and \cref{CountingBSD} tells us that
\[
|\BSD(2n \times n,n)| = G \setminus \left( \bigcup_{r=1}^n \bigcup_{e_1,\dotsc, e_r \in E} G(e_1,\dotsc,e_r) \right).
\]
Using the inclusion-exclusion principle, it follows that
\[
|\BSD(2n \times n,n)| = |\BST(2n \times n,n)|-\sum_{e_1\in E} |G(e_1)| + \sum_{e_1,e_2\in E}  |G(e_1,e_2)|- \dotsb 
\]
We then observe that these graphs are characterized by 
the connected components induced by the forced edges $e_1,\dotsc,e_r$, determining a partition $p$ of $n$.
Furthermore, the sign in the above formula only depends on the number of forced edges, 
which is equal to $|p-1|$, so we can transform this into a sum over all partitions of $n$.
Given a partition $p \vdash n$, the number of graphs with component sizes $p_1,p_2,\dotsc,p_k$
is given by $\frac{1}{m!} \binom{|p|}{p}$, with $m$ given as above.

\noindent
\emph{Claim:} Let $e_1,e_2,\dotsc,e_r$ be fixed edges such that the component sizes are given by $p$.
Then 
\[
G(e_1,\dotsc,e_r) =  \binom{|p + 1|}{p + 1}.
\]
\emph{Proof:} 
Suppose $\Gamma_\sigma \in G(e_1,\dotsc,e_r)$ has a component $(i_1,i_2,\dotsc,i_j)$, in increasing order.
From \cref{Injection_to_S_m}, we know $\sigma^{-1}(i_s)<\sigma^{-1}(i_s+n)$ for 
all $1\leq s\leq j$, and for $(i_1,i_2,\dotsc,i_j)$ to be connected, 
we need $i_s +n$ to form an $n$-descent with $i_{s-1}$ for 
all $1<s\leq j$ \emph{i.e.} $\sigma^{-1}(i_s+n)=\sigma^{-1}(i_{s-1})-1$.

Together these two statements imply that $\sigma$, has the following structure:
\begin{align*}
 \sigma^{-1}(i_j)<\sigma^{-1}(i_j+n) \lessdot \sigma^{-1}(i_{j-1}) < \sigma^{-1}(i_{j-1}+n) \lessdot \dotsb \\
\dotsb  <\sigma^{-1}(i_{3}+n)\lessdot \sigma^{-1}(i_{2}) < \sigma^{-1}(i_{2}+n) \lessdot \sigma^{-1}(i_{1})<
\sigma^{-1}(i_{1}+n)
\end{align*}
where $a \lessdot b$ means that $a+1=b$. Thus, we have $j+1$ \emph{blocks}, 
\[
 [\sigma^{-1}(i_j)],\;  
 [\sigma^{-1}(i_{s+1}+n) \lessdot \sigma^{-1}(i_{s})] \text{ for $s=1,\dotsc,j-1$}
 \text{ and } [\sigma^{-1}(i_1+n)]
\]
which need to appear in order, but there is no further restriction.
The number of $\Gamma_\sigma$ with component 
sizes determined by $p$ is therefore $\binom{|p + 1|}{p + 1}$, which concludes the proof.
\end{proof}

Let us now dive into a completely different part of mathematics.
The \emph{Weil--Petersson volume}, $\Vol_{WP}(\cdot)$,
is defined as
\[
\Vol_{WP}(M) \coloneqq \frac{1}{(n-3)!}\int_{M} \wedge^{n-3}(\omega_{M}).
\]
where $\omega M$ is the Weil--Petersson symplectic form on $M$.
\medskip

Let $M_{0,n}$ be the moduli space of an $n$-punctured Riemann sphere, that is
\[
 M_{0,n} \coloneqq \{ (z_1,\dotsc,z_n) \in \hat{\setC}^n\ : z_i \neq z_j\} / \symS_n \times \mathrm{PSL}(2,\setC) 
\]
and $\symS_n$ acts by permuting variables, and $\mathrm{PSL}(2,\setC)$ acts as a linear fractional transformation.

\begin{theorem}[\cite{Zograf1993}]
The Weil--Petersson volume of the moduli space of an $n$-punctured Riemann sphere $M_{0,n}$ is given by the formula
\[
\Vol_{WP}(M_{0,n}) = \frac{\pi^{2(n-3)} v_{n}}{n!(n-3)!}, \text{ for $n\geq 4$},
\]
where the sequence $v_n$ for $n\geq 3$, $v_3=1$ be defined via the recursion
\begin{align}\label{eq:zograf}
v_n = \frac{1}{2}\sum_{i=1}^{n-3}
\frac{ i(n - i -2)}{n-1}
\binom{n - 4}{i - 1} \binom{n}{i + 1} v_{i +2} v_{n - i}, \qquad n\geq 4.
\end{align}
\end{theorem}
This sequence shows up as \texttt{A115047} in the OEIS, \cite{OEIS},
see \cite{Kaufmann1996,Matone2001} for more background.

In \cite[Equation (0.7)]{Kaufmann1996}, the following relationship is shown
\begin{proposition}\label{prop:KaufmannFormula}
Let the sequence $v_n$ be defined as in \eqref{eq:zograf}.
Then 
\begin{equation*}
v_n = \sum_{k=1}^{n-3} \frac{(-1)^{n-3-k}}{k!} \!\!\!\!
\sum_{\substack{m_1,\dotsc,m_k > 0 \\ m_1+\dotsb+m_k=n-3 }} \binom{n-3}{m_1,\dotsc,m_k} \binom{n-3+k}{m_1+1,\dotsc,m_k+1}.
\end{equation*}
\end{proposition}

We are now ready to prove the following connection
between the sequence $v(n)$ and border-strip decompositions:
\begin{theorem}\label{thm:2nxn}
Let
\begin{equation}\label{eq:2nxnRecurrence}
a(n) = \frac{1}{2} \sum_{i=1}^n \frac{ i(n - i + 1)}{ (n + 2)} \binom{n - 1}{i - 1} \binom{n + 3}{i + 1} a(i - 1) a(n - i) .
\end{equation}
and let $v_n$ be given as in \eqref{eq:zograf}. Then $a(n) = v_{n+3} =  |\BSD(2n \times n ,n)|$.
\end{theorem}
\begin{proof}
The first equality, $v_{n+3} = a(n)$ follows from comparing \eqref{eq:zograf} and \eqref{eq:2nxnRecurrence}.
It is a straightforward calculation to verify that these are equal.

To get the second identity, note that we can get the formula in \cref{prop:KaufmannFormula}
from \cref{eq:sumAsGraphs} 
by replacing partitions with compositions, and then refining the sum over 
the number of parts (denoted $k$ in \cref{prop:KaufmannFormula}). 
\end{proof}

\section{Further directions}

Given the connection with Euler characteristics of 
moduli spaces mentioned after \cref{cor:totalNumberOfBorderStripDecomps},
and the connection with moduli spaces in \cref{thm:2nxn}, is there 
a generalization of this mysterious connection? For example, there are formula for the 
volumes of surfaces of other genus, see \cite{Matone2001}.

% \todo{Can we define something similar to Eulerian polynomials?
% %https://ac.els-cdn.com/S0196885810001284/1-s2.0-S0196885810001284-main.pdf?_tid=0bce122c-e64a-11e7-8533-00000aab0f6c&acdnat=1513859316_8ba1be50db428fe3bdb6cace0ccfd105
% What about major index?
% }

Another interesting direction is to consider the $q$-analogue of border-strip tableaux rather than decompositions.

\subsection*{Acknowledgements}

The authors are thankful for the observation made by Ziqing Xiang in \cite{ZiqingXiang294267}.
We also thank Richard Stanley for pointing out the reference \cite{Pak2000Ribbon},
and Justin Troyka for pointing out an error in an earlier draft. 

The first author is funded by the \emph{Knut and Alice Wallenberg Foundation} (2013.03.07).

\bibliographystyle{alpha}
\bibliography{bibliography}

\end{document}